\documentclass{article}
\usepackage{amsmath, amssymb, amsthm}

\makeatletter
\newcommand{\subjclass}[2][1991]{%
  \let\@oldtitle\@title%
  \gdef\@title{\@oldtitle\footnotetext{#1 \emph{numbers\thinspace:} #2}}%
}
\newcommand{\keywords}[1]{%
  \let\@@oldtitle\@title%
  \gdef\@title{\@@oldtitle\footnotetext{\emph{Key words and phrases.} #1.}}%
}
\makeatother

\DeclareMathOperator{\dm}{\overline{\mathcal{M}}}
\DeclareMathOperator{\m}{\mathcal{M}}
\DeclareMathOperator{\uc}{\mathcal{C}}
\DeclareMathOperator{\scr}{\overline{\mathcal{C}}}
\DeclareMathOperator{\sdiv}{\partial \dm}
\DeclareMathOperator{\mo}{\text{Mod}}

\DeclareMathOperator{\xreg}{X_{\text{reg.}}}
\DeclareMathOperator{\xsing}{X_{\text{sing.}}}
\DeclareMathOperator{\sreg}{S_{\text{reg.}}}
\DeclareMathOperator{\ssing}{S_{\text{sing.}}}

\theoremstyle{plain}
\newtheorem{lem}{Lemma}[section]
\newtheorem{prop}{Proposition}[section]
\newtheorem{thm}{Theorem}[section]
\newtheorem*{thm*}{Theorem}

\newtheorem{theorem}{Theorem}

\theoremstyle{remark}
\newtheorem*{rem}{Remark}
\theoremstyle{definition}
\newtheorem{defin}{Definition}[section]

\title{Achiral Lefschetz fibrations and the moduli space of curves}
\author{Sardor Yakupov}
\date{\today}

\begin{document}
\maketitle
\begin{abstract}
Symplectic Lefschetz fibrations can be described via classifying maps with values in the Deligne--Mumford compactification of the moduli space of curves, by means of constructions relying on symplectic geometry. In this note we prove the existence of classifying maps for achiral Lefschetz fibrations using Riemannian geometry. We further extend Smith's signature formula to the achiral case, providing a more elementary proof of this statement.
\end{abstract}
\tableofcontents
\newpage
\section{Introduction}
Achiral Lefschetz fibrations constitute a powerful tool in the study of 4-manifolds and form a natural generalization of Lefschetz fibrations, which were in turn inspired by analytic families of curves. \par 
We denote by $\m_g$ the moduli space of smooth algebraic curves of genus $g$ and its compactification by $\dm_g$. It admits the universal curve $\scr_g$ with fiber above a point $x\in \dm_g$ given by the curve with the isomorphism class represented by $x$. It exhibits a universal property for analytic families of stable curves: to a family $C\to S$ (satisfying minor additional conditions) is associated a natural map $\phi:S\to \dm_g$, such that the family $C$ is isomorphic to the pullback of $\scr_g$ via $\phi$.\par 
In 1999 Smith described \cite{Smith1}, using symplectic geometry, how to construct a classifying map $\phi: S\to \dm_g$ associated to a symplectic Lefschetz fibration $f:X\to S$, where $X$ and $S$ are closed connected oriented manifolds of dimension 4 and 2 respectively. Moreover, he showed that the signature $\sigma(X)$ can be calculated in terms of this classifying map. \par 
In this article we present a generalization of this construction for the case of achiral Lefschetz fibrations:
\begin{theorem}[see Theorem~\ref{thm:main1}]
    Let $f:X\to S$ be an achiral Lefschetz fibration of fiber genus $g\geq 4$. Then there exists a classifying map $\phi_f: S\to \dm_g$, which is unique up to a choice of slicing metrics and homotopy of the regular part. Moreover, if $f$ does not admit $(1,g-1)$-type separating singularities, we can arrange the image of $\phi_f$ to be contained in the manifold $\dm_g^o$, the automorphism-free locus of $\dm_g$, and $\phi_f^*(\scr_g) \simeq X$.
\end{theorem}
Since the symplectic approach fails in the achiral case, we prove this result using Riemannian geometry. We start by prescribing the metric to ensure the smoothness of $\phi_f$ around the singular locus, and then proceed by extending and uniformizing the metrics on each fiber $F_x= f^{-1}(x)$. The tool we introduce for this is the \emph{normalized fibered Ricci flow} (NFRF), which is a straightforward generalization of the Ricci flow to the fiber bundle case. We prove a convergence result similar to that of the surface case treated in~\cite{Ham1}:
\begin{theorem}[see Theorem~\ref{NFRF}]
    Let $f:X\to S$ be a bundle with fibers being closed surfaces of genus $g\geq 2$ over a compact base $S$,
    equipped with an arbitrary smooth family of Riemannian metrics $g_{0,p}$. Then there exists a unique solution $(g_{t,p})_{t\in[0;+\infty)}$ of the normalized fibered Ricci flow that converges to a smooth family $g_{\infty,p}$ of Riemannian metrics of constant curvature.
\end{theorem}
Finally, in Section~\ref{section:signature} we explain how the Smith signature formula, slightly modified, generalizes to the achiral case:
\begin{theorem}[see Theorem~\ref{thm:main2}]
    Let $f:X\to S$ be an achiral Lefschetz fibration of fiber genus $g\geq 4$ whose classifying map is $\phi_f : S\to \dm_g$. Then, the signature $\sigma(X)$ of $X$ can be expressed as:
    $$\sigma(X) = \langle \hat\sigma, [\phi_f]\rangle$$
    where $\hat{\sigma} = 4c_1(\lambda)-PD(\sdiv_g)\in H^2(\dm_g,\mathbb{Q})$, $c_1(\lambda)$ is the first Chern class of the Hodge bundle $\lambda$ and $PD$ is the Poincar\'e dual of the orbifold fundamental class of the singular divisor $\sdiv_g$.
\end{theorem}
The original proof of Smith was based on the index theorem of Atiyah and Singer applied to the signature operator on fiber bundles with boundary correction term (see, for example, \cite{SignatureofFibreBundles} and \cite{spectralasymm}). We provide a more elementary proof of the signature formula, which uses the index theorem in a more indirect way, through the existence of the Meyer signature cocycle $\tau_g$.
\subsection*{Acknowledgments}
I would like to thank my Ph.D. supervisor Louis Funar and Gérard Besson for helpful discussions and encouragements.
I would also like to thank Marco Golla, Patrick Popescu-Pampu, Alberto Verjovsky and Emmanuel Giroux for their feedback on the preliminary versions of this paper.
\section{Preliminaries}
\subsection{Achiral Lefschetz fibrations}
In this subsection $X$ denotes a closed connected oriented 4-manifold, $S$ and $F$ denote closed connected oriented surfaces of genus $h$ and $g$, respectively. We will often restrict ourselves to the case where the base surface $S$ is the 2-sphere $S^2$. \par 
\begin{defin}
    We say that a smooth map $f:X\to S$ admits a \emph{(chiral) Lefschetz type singularity} at $x\in X$ if there exist orientation-preserving local charts centered at $x$ and $f(x)$ such that $f$ takes the form $\begin{cases}
        \mathbb{C}^2 \to \mathbb{C}\\
        (z,w) \mapsto zw
    \end{cases}$. \par 
    Furthermore, if the local form at $x$ is given by $(z,w)\mapsto \bar{z}w$, we say that $f$ admits an \emph{achiral Lefschetz type singularity}.
\end{defin}
    \begin{rem}
        Note that the charts we are considering are smooth and not necessarily holomorphic. In particular, we do not suppose the existence of an (almost) complex structure on $X$. 
    \end{rem}
\begin{defin}
    We call a smooth map $f:X\to S$ an \emph{achiral Lefschetz fibration}, if each critical point of $f$ is of Lefschetz or achiral Lefschetz type. Moreover, if $f$ only admits chiral Lefschetz singularities, we call it a \emph{Lefschetz fibration}, or a chiral Lefschetz fibration if we want to draw a clear distinction to the achiral case. \par 
    Denote the critical set by $\text{Crit}=\{x_1,\dots,x_n\}$ (it is finite, since $X$ is compact and the critical points are isolated) and $p_i=f(x_i)$ the critical values of $f$, respectively.
\end{defin}
\begin{rem}
    The map $f$ can always be perturbed to become injective on the critical set $\text{Crit}$, so we will only consider such fibrations in the sequel. \par 
    Note also that if $f$ only admits achiral Lefschetz singularities, we can see it as a chiral Lefschetz fibration by considering the opposite orientation on the total space $X$.
\end{rem}
Note that outside the critical values the restriction $f|_{X\setminus f^{-1}\{p_1,\dots,p_n\}} : X\setminus f^{-1}\{p_1,\dots,p_n\} \to S\setminus \{p_1,\dots,p_n\}$ is a smooth bundle. Denote its fibers by $F$. We will refer to $F$ as the \emph{regular fiber} of $f$, whereas the preimages $f^{-1}(p_i)$ will be called \emph{singular fibers}. We will restrict our attention to the case where the regular fiber $F$ is connected, which is always the case when $S=S^2$. This is due to the well-known long exact sequence in $\pi_1$ and $\pi_0$ of the fibers, the total space and the base for (achiral) Lefschetz fibrations (see \cite[Proposition 8.1.9]{GompfStipsicz}): 
$$
\pi_1(F)\to\pi_1(X)\to\pi_1(S)\to\pi_0(F)\to 1.
$$
\par
The topology of singular fibers can be described using the notion of a \emph{vanishing cycle}. For every critical value $p_i$, there exists a simple closed curve $\gamma_i$ on $F$ such that the singular fiber $f^{-1}(p_i)$ is obtained from $F$ by collapsing $\gamma_i$ to a point. Moreover, the monodromy $\tau_i$ around a Lefschetz-type singular value $p_i$ is given by a right-handed Dehn twist along $\gamma_i$, and by a left-handed twist in the achiral case. For $g\geq 2$, the monodromy representation $\pi_1(S\setminus \{p_1,\dots, p_n\})\to \mo_g$ determines the achiral Lefschetz fibration up to fiber-preserving diffeomorphism. For $S=S^2$ this is the monodromy factorization up to Hurwitz moves and global conjugation. \par 
Lefschetz fibrations of dimension 4 have been extensively studied; see, for example, \cite[Chapter 8]{GompfStipsicz} or \cite[Chapter 7]{Akbulut} for an overview. \par 
We want to mention two major results in the study of chiral Lefschetz fibrations by Gompf \cite{GompfStipsicz} and Donaldson \cite{DonaldsonPencils}, respectively:
\begin{thm}[Theorem 10.2.18, \cite{GompfStipsicz}]
    Let $f:X\to S$ be a Lefschetz fibration such that the homology class of the fiber $[F]$ does not vanish in $H_2(X,\mathbb{R})$. Then $X$ admits a symplectic structure $\omega$ such that the regular fibers $F_x = f^{-1}(x)$ are symplectic.
\end{thm}
\begin{thm}[\cite{DonaldsonPencils}]
    Let $(X,\omega)$ be a closed symplectic 4-manifold. Then there exists a number $n\in \mathbb{N}$ such that $X\#^n \overline{\mathbb{C}P^2}$ admits a Lefschetz fibration.
\end{thm}
These two results establish a deep connection between symplectic manifolds and Lefschetz fibrations of dimension $4$. \par 
For achiral Lefschetz fibrations, there are known inequalities relating Betti numbers $b_1$ and $b_2$ (see, for example, \cite[Theorem 8.4.13]{GompfStipsicz}). These provide obstructions to the existence of achiral Lefschetz fibrations on manifolds with sufficiently large fundamental group. By contrast, much less is known about the simply connected manifolds that support achiral Lefschetz fibrations.
\subsection{Moduli space of complex curves}
In this subsection $F$ still denotes a closed connected oriented surface of genus $g$. Recall that $\m_g$ is the space of complex structures on $F$ up to biholomorphism. It is well-known that $\m_g$ is a complex algebraic variety and actually a complex orbifold of dimension $3g-3$ for $g\geq 2$.\par 
The universal curve $\uc_g$ is a natural bundle over $\m_g$ in the orbifold sense. For a point $x\in\m_g$ represented by the class of a Riemann surface $C_x$, the fiber above $x$ is precisely the complex curve $C_x$. Similarly to the moduli space, the universal curve is a complex algebraic variety, and a complex orbifold of dimension $3g-2$. This bundle satisfies the following important universal property in the analytic category: every family $C\to B$ of smooth complex curves of genus $g$ is isomorphic to the orbifold pullback of $\uc_g$ via the naturally induced morphism $\phi :B \to \m_g$ (the pullback can be taken in the usual sense if the family does not admit curves with non-trivial automorphisms). The morphism $\phi: B\to \m_g$ is called the \emph{classifying map} for the family $C\to B$.\par 
The Hodge bundle $\lambda_g$ is a vector bundle over $\m_g$ whose fiber over a point $x\in \m_g$ is given by the space of holomorphic differentials on $C_x$. \par 
We should also mention that the cotangent bundles $T^*\m_g$ and $T^*\uc_g$ admit nice descriptions, despite $\m_g$ being an orbifold. The cotangent space at $x=[C_x]\in\m_g$ is given by the space of holomorphic quadratic differentials on $C_x$, and the cotangent space at $(x,p) = ([C_x], p\in C_x) \in \uc_g$ is given by the space of meromorphic quadratic differentials on $C_x$ with at most a simple pole at $p$ and holomorphic elsewhere; the interested reader can consult \cite{ArbarelloCornalbaGriffiths} for a more detailed description. \par 
Finally, $\m_g$ is endowed with a Hermitian metric, called the \emph{Weil--Petersson} metric. With the above description of cotangent space, it is given by 
    \begin{align*}
        &WP: T^*_x \m_g \times T^*_x \m_g \to \mathbb{C} \\
        &\alpha, \beta \mapsto \int_{C_x} \alpha \bar{\beta} (g_x)^{-1}
    \end{align*}
    where $g_x = \zeta dz d\bar{z}$ is the metric tensor of the unique hyperbolic metric compatible with the Riemann surface structure $C_x$.
Moreover, the Weil--Petersson metric is well defined on the cotangent space of the universal curve, taking $g_{x,p}$ to be the uniformized metric associated to the Riemann surface $C_x$ with a cusp at $p$. More details on the Weil--Petersson metric can be found in the survey paper \cite{Wolpert} by Wolpert. \par 
Although the space $\m_g$ is not compact, it has a natural compactification. A point on an algebraic curve $C$ over $\mathbb{C}$ is called a \emph{node} if it has a neighbourhood biholomorphic to the subset $\{zw=0\}$ of $\mathbb{C}^2$. We say that a curve $C$ is \emph{nodal}, if every point of $C$ is either smooth or nodal. If, moreover, $C$ is connected and of arithmetic genus $g\geq 2$ and the automorphism group of $C$ is finite, we say that it is \emph{stable}. \par
We denote by $\dm_g$ the space of isomorphism classes of genus $g$ stable curves. This space is a \emph{Deligne--Mumford stack}, and thus a complex orbifold. It naturally contains $\m_g$, and it was proven by Deligne and Mumford \cite{DM} that $\m_g$ is open and dense in $\dm_g$. Moreover, the coarse moduli space of $\dm_g$ is a compact projective algebraic variety \cite{Knudsen} and $\dm_g$ is called the \emph{Deligne--Mumford compactification} of $\m_g$. \par 
The universal curve extends to the compactification, although not as a bundle anymore. We will denote the extension $\scr_g$. It admits a similar universal property in the analytic category, but for families of stable curves this time (see Theorem 12.1 in \cite{HUbbardKoch}). \par 
The Hodge bundle extends to $\dm_g$ as a rank $g$ vector bundle $\pi_*\omega_{\scr_g/\dm_g}$, whose fiber is the space of sections of the dualizing sheaf of the stable curve. In particular, its first Chern class is defined in $H^2(\dm_g,\mathbb{Q})$. Using a slight abuse of notation, we still denote this extension by $\lambda_g$ and its first Chern class $c_1(\lambda_g)\in H^2(\dm_g,\mathbb{Q})$. \par 
Note that the Weil--Petersson metric does not extend to a smooth metric on $\dm_g$. It is incomplete and degenerates along $\sdiv_g$ \cite{Masur}, although the completion of $\m_g$ with respect to this metric is homeomorphic to $\dm_g$. The fundamental paper \cite{HUbbardKoch} by Hubbard and Koch provides fine details of the geometry of the Deligne--Mumford compactification and in particular its cotangent spaces. \par 
We will introduce a notion that will be useful later on:
\begin{defin}
    Let $f: X\to D^2$ be a genus $g$ achiral Lefschetz fibration over a disk with a single singularity $x_s$ with monodromy given by $\tau\in \mo_g$. We say that a Riemannian metric $h$ on $X\setminus \{x_s\}$ is \emph{slicing} for $\tau$ if the fibers $(F_y, h|_{F_y})$ are hyperbolic surfaces of constant Gauss curvature $-1$ for every $y\in D\setminus f(x_s)$.
\end{defin}
\begin{rem}
    The intuition for the adjective ``slicing'' can be seen from the following picture. A choice of a slicing metric corresponds essentially to a choice of an embedded disk $D \subseteq \dm_g$ which intersects the singular divisor corresponding to the monodromy $\tau$ transversely at a single point (i.e. ``slices'' it).
\end{rem}
It turns out that such metrics always exist:
\begin{prop}\label{slicingprime}
    Let $\gamma$ be an essential simple closed curve on $\Sigma_g$. Then the Dehn twists $\tau_{\gamma}$ and $\tau_{\gamma}^{-1}$ admit a slicing metric.
\end{prop}
\begin{proof}
        Let $C_0$ be a stable curve with a single node, of the topological type obtained by pinching $\gamma$. Let $V=\Delta_t\times W$ (where $\Delta_t$ is a complex disk parametrizing the nodal deformation and $W$ parametrizes locally trivial deformations) be a plumbing chart of $\dm_g$ at $[C_0]$, with universal family $\mathcal{C}\to V$ given by $xy=t$ near the node, and $q: V\to[V/\text{Aut}(C_0)]\to\dm_{g}$. Then $q^{-1}(\sdiv_g)=\{t=0\}$. For $w_0\in W$, we define $\Tilde{\iota}(z)=(z,w_0)$ for $\tau_\gamma$ and $\Tilde{\iota}(z)=(\bar{z},w_0)$ for $\tau_\gamma^{-1}$. The pullback $\Tilde{\iota}^*\mathcal{C}\to D^2$ is $\{xy=z\}$ (resp. $\{xy=\bar{z}\}$) near the node and locally trivial elsewhere. It is therefore a chiral (resp. achiral) Lefschetz fibration with vanishing cycle $\gamma$, and by the Picard--Lefschetz formula its monodromy is given by $\tau_\gamma^{\pm1}$. Give every smooth fiber its hyperbolic metric. These metrics vary smoothly for $z\in D^*$ and converge in $\mathcal{C}^\infty_{loc}$ on the central fiber away from the node, by Wolpert's expansion of the hyperbolic metric in plumbing coordinates \cite{Wolpert2}. Adding any horizontal component gives a slicing metric.\par
        Moreover, if $\gamma$ is not separating of type $(1,g-1)$ and $g\geq 3$, the curve $[C_0]$ can be chosen to be automorphism-free; this will be useful for the automorphism-free part of Theorem \ref{thm:main1}.

\end{proof}
\subsection{Ricci flow for surfaces}
In this subsection we quickly remind the reader about the Ricci flow on surfaces and its use in a proof of the uniformization theorem. Most of the results were obtained by Hamilton (\cite{Ham3} and \cite{Ham1}) with a substantial improvement of short-time existence result by DeTurck \cite{DeTurck}.
These results were rewritten in a more accessible way in multiple textbooks. Here we will mostly refer to the textbook \cite{ChowKnopf} by Chow and Knopf.
\begin{defin}
    Let $(F,g_t)_{t\in [0; T)}$ be a closed connected genus $g$ oriented surface $F$ equipped with a family of smooth Riemannian metrics $g_t$ with a $C^1$-dependence on the parameter $t$. We say that this family is a solution to the \emph{normalized Ricci flow} with the initial condition $g_0$ if 
    $$\partial_t g_t = (k-K_{g_t}) g_t$$
    where $k$ is a constant and $K_{g_t}$ is the scalar curvature of $g_t$.
\end{defin}
The constant $k$ plays the role of the desired uniform curvature, so it should respect the Gauss--Bonnet formula (i.e. $k\cdot vol(F) = 4\pi\chi(F)$). As we consider the high genus ($g\geq 2$) case, often $k$ will be $-2$. \par 
In general, the Ricci flow has good short-time behaviour in any dimension:
\begin{thm}[\cite{Ham3}, \cite{DeTurck}]
    Let $(M, g_0)$ be a closed $n$-dimensional Riemannian manifold. Then there exists $T>0$ and a unique solution $(g_t)$ of the normalized Ricci flow on $0\leq t < T$ with the initial condition $g_0$.
\end{thm}
The long-time behaviour is much more complicated in general, as the flow can form singularities. However, in the surface case the flow is singularity-free:
\begin{thm}[\cite{Ham1}, \cite{Chow}]
    Let $(F, g_0)$ be a closed genus $g$ surface. Then there exists a solution $(g_t)$ of the normalized Ricci flow on $[0;\infty)$. Moreover, the family $g_t$ converges in $\mathcal{C}^\infty$ topology to the constant curvature $k$ metric $g_\infty$ conformally equivalent to $g_0$.
\end{thm}
In the high genus ($g\geq 2$) case, the proof is relatively simple and can be consulted in \cite[Chapter 5]{ChowKnopf} or in the original article \cite{Ham1}. We underline the importance of understanding this proof, as very similar techniques will be used in the main section. We will mention here the most important technical tool in the proof, the \emph{maximum principle}:
\begin{prop}[Maximum principle]\label{maxprinc}
    Let $(M, g_t)$ be a closed $n$-manifold equipped with a smooth family of Riemannian metrics. Consider $u:M\times [0,T) \to \mathbb{R}$ a smooth function satisfying 
    $$\partial_tu - \Delta_{g_t}u \leq g_t(V(t), \nabla^{g_t}u)+F(u)$$
    for some vector field $V(t)$ and locally Lipschitz function $F$. Then for any constant $c$ such that $u_{t=0} \leq c$, we have $u(x,t)\leq U(t)$, where $U$ is the solution to the ODE $\begin{cases}
        \frac{d}{dt}U=F(U)\\ U(0)=c
    \end{cases}$.
\end{prop}
\begin{rem}
    \begin{enumerate}
        \item In all of our applications, we will drop the $g_t(V(t), \nabla^{g_t}u)$ term by setting $V=0$, while $F(u)$ will represent nonlinear terms in evolution equations for curvature and similar quantities.
        \item If $u$ satisfies the opposite inequality, we can apply the maximum principle to $\Tilde{u}:=-u, \Tilde{F}(u):=-F(-u)$ and $\Tilde{V}:=V$ to obtain the bound $u(x,t) \geq -\Tilde{U}$. When using the result in this way, we will call it \emph{minimum principle}.
        \item The name and the intuition behind the maximum principle come from physics, and more particularly from thermodynamics. The intuition to have in mind is that $\partial_t - \Delta_{g_t}$ is a heat propagation operator, $u$ and $c$ play the role of temperature distribution and maximal temperature respectively, while $F(u)$ can be considered as a nonlinear part of the heat flow. The principle then asserts that an evolution of a temperature distribution is always bounded by an evolution of a uniformly maximal temperature.
    \end{enumerate}
\end{rem}
\section{Construction of the classifying map}
In this section we will present a construction of a smooth classifying map for an arbitrary achiral Lefschetz fibration of fiber genus $g\geq 4$. Moreover, the construction is done in such a way that the resulting map will be well defined up to homotopy of the regular part. \par
To ensure that the constructed map is smooth near $\sdiv_g$, we shall start by paying attention to the singular points. Choose disjoint open disks $D_i$ around each singular value $p_i\in S$ and denote the union of these disks $\ssing := \bigsqcup D_i$. Let $D'_i$ be an open disk around $p_i$ that satisfies $\overline{D'_i}\subseteq D_i$ and denote $\sreg = S\setminus \bigsqcup D'_i$. Note that we have left a small overlap between $\sreg$ and $\ssing$, it is done purely for technical reasons to ensure smoothness while gluing. We will denote the corresponding fibered neighbourhoods $\xsing := f^{-1}(\ssing)$ and $\xreg:=f^{-1}(\sreg)$. \par 
We would like to start by putting the slicing metric on the singular part $\xsing$, and then extend it to a Riemannian metric defined on the whole $X$. In order to do so, we will use the following lemma:
\begin{lem}\label{slicing}
    Let $f:X\to S$ be an achiral Lefschetz fibration and $p_1,\dots,p_n$ be the set of its singular values with vanishing cycles $\gamma_1,\dots,\gamma_n$. Let $g_{\gamma_1},\dots,g_{\gamma_n}$ be a choice of slicing metrics for the corresponding monodromies $\tau_{\gamma_i}^{\pm 1}$. Then the space of smooth Riemannian metrics $g$ on $X\setminus\text{Crit}$ such that each singular value $p_i$ admits a disk neighbourhood $V_i$ with the restriction of $g$ to $f^{-1}(V_i)$ equal, after fixing an identification of $f^{-1}(V_i)$ with the model family over the slicing disk, to the chosen slicing metric $g_{\gamma_i}$, is contractible.
\end{lem}
\begin{proof}
    We start by showing that the space of such metrics is not empty. Recall that a Riemannian metric on $X\setminus\text{Crit}$ is a section of the bundle $\mathcal{G}$ of positive-definite bilinear forms on the tangent spaces, which is a subbundle of $(T^*X)^{\otimes 2}$. The fibers of the bundle $\mathcal{G}$ are contractible, as they are convex. Therefore, the existence of a metric $g$ with desirable properties boils down to the problem of extending a section of a bundle defined in a subset of $X$. As we can choose our neighbourhoods to be closed, the standard results of section extensions apply (see \cite[Section 29]{Steenrod}). \par 
    The contractibility of this space can now be deduced from general results on bundles with contractible fibers (see \cite[Section 34]{Steenrod}), but it also follows directly from the convexity of space of Riemannian metrics.
\end{proof}
This lemma provides us with a smooth classifying map defined on $\ssing$, as the chosen slicing metrics ensure good local behaviour. We shall now focus on the regular part $\sreg$. Note that on the regular part $f:\xreg \to \sreg$ is a submersion and thus a fiber bundle. To construct the classifying map we would like to associate the conformal class of the metric $g|_{F_p}$ restricted to the fiber $F_p=f^{-1}(p)$ with each point $p\in \sreg$. Although such a map is known to be smooth (see \cite{EE}), we would like to have a more explicit, geometric way to construct this map. To this end, we would like the metrics $g|_{F_p}$ to be already uniformized and have the same constant curvature so that the classifying map takes a simpler form. In order to do so, we will introduce the tool called \emph{normalized fibered Ricci flow}:
\begin{defin}
    Let $f:X\to S$ be a fiber bundle equipped with a smooth two-parameter family of Riemannian metrics $(g_{t,p})$ defined on fibers of $f$ with $t\in [0;T), p\in S$. Concretely, it means that for every $(t,p)$, $g_{t,p}$ is a Riemannian metric on the fiber $F_p = f^{-1}(p)$. \par 
    We say that this family is a solution to the \emph{normalized fibered Ricci flow} (NFRF) if for every $p\in S$, $t\mapsto g_{t,p}$ is a solution to the normalized Ricci flow on the fiber $F_p$. 
\end{defin} 
\begin{rem}
    There is an already existing notion of \emph{laminated Ricci flow} and some results in the case of surface laminations (cf. \cite{MunizVerjovsky} and \cite{Verjovsky}). Unfortunately, given the generality of the framework of laminations, these results were not sufficient to construct a smooth classifying map.
\end{rem}
\begin{thm}\label{NFRF}
    Let $f:X\to S$ be a surface bundle with fiber genus $g\geq 2$ over a compact base, equipped with an arbitrary smooth family of Riemannian metrics $g_{0,p}$. Then there exists a unique solution $(g_{t,p})_{t\in[0;+\infty)}$ to the normalized fibered Ricci flow with this initial data. Moreover, this family converges smoothly to a family $g_{\infty,p}$ of constant curvature metrics.
\end{thm}
To prove this theorem, we need to address the following questions:
\begin{enumerate}
    \item smoothness in the $p$-variable and uniqueness of solutions to the fibered Ricci flow,
    \item long-time existence and convergence to the constant curvature metric family.
\end{enumerate}
The classical Ricci flow theory answers both of these questions fiberwise, so in reality we only need to check that the family $g_{t,p}$ (given by the fiberwise solution to the normalized Ricci flow) is smooth in the $p$-variable and the convergence to the family $g_{\infty, p}$ is uniform in the $p$-variable. \par 
The first question follows from DeTurck's trick, since it makes the Ricci flow strictly parabolic, and thus the solutions depend smoothly on the parameters (e.g. see \cite[Chapter 3]{ChowKnopf}). Uniqueness also follows from the harmonic map flow step in DeTurck's argument. \par 
The second question requires a surface-specific approach, but it turns out that a careful rewriting of the surface case proof together with compactness of $S$ yields the result.  We start by recalling a well-known result for metric flows:
\begin{prop}
    Let $(M^n,g_t)_{t\in[0;T)}$ be a smooth one-parameter family of metrics on a closed manifold $M^n$, where $T\leq \infty$. If the following conditions 
    $$\int \underset{M}{\sup} |\nabla^i \partial_t g_t|_{g_t}dt <\infty$$
    hold for all $i$, then the family $g_t$ converges to a smooth metric $g_T$ in the $\mathcal{C^\infty}$ topology.
\end{prop}
In the surface case this means that we need to bound the curvature uniformly over all fibers to conclude. We will proceed in a series of simple lemmas:
\begin{lem}
    Let $g_{t,p}$ be a solution to the NFRF with surface fibers of genus $g\geq 2$. Then there exists a constant $C'$ such that for each $(t,p)$ we have:
    $$K_{g_{t,p}}-k \geq C' e^{kt}$$
\end{lem}
\begin{proof}
    Under the NFRF the curvature $K_{g_{t,p}}$ satisfies the following evolution equation for each $p\in S$:
    $$\partial_t K_{g_{t,p}} = \Delta_{g_{t,p}}K_{g_{t,p}} + K_{g_{t,p}} (K_{g_{t,p}}-k)$$
    Applying the minimum principle (see remark after proposition~\ref{maxprinc}) to this equation using $V(t)=0$ and $F(U)=U(U-k)$ gives the following fiberwise inequality:
    $$K_{g_{t,p}}-k \geq \frac{-k}{\frac{\underset{x\in F_p}{\min} K_{g_{0,p}}(x)-k}{\underset{x\in F_p}{\min} K_{g_{0,p}}(x)}e^{kt}-1}-k \geq C'_{p} e^{kt}$$
    with $C'_p = \underset{x\in F_p}{\min} K_{g_{0,p}}(x)-k$. The last inequality is deduced using $\underset{x\in F_p}{\min} K_{g_{0,p}}(x) < 0$, which follows from Gauss--Bonnet for $g\geq 2$. As the curvature $K_{g_{0,p}}$ varies smoothly between the fibers, the compactness of $S$ ensures that the global minimum $C':=\underset{p}{\min} C'_p$ is well defined.
\end{proof}
We obtain the upper bound in a similar, albeit slightly trickier, manner:
\begin{lem}
    Let $g_{t,p}$ be a solution to the NFRF with surface fibers of genus $g\geq 2$, then there exists a constant $C_0$ such that for each $(t,p)$ we have:
    $$|K_{g_{t,p}}(x)-k| \leq C_0 e^{kt}$$ 
\end{lem}
\begin{proof}
    For each $p\in S$ we can solve the time-dependent Poisson equation 
    $$\Delta_{g_{t,p}} \Phi_p(t) = k-K_{g_{t,p}}$$
    and the solution is well-defined up to a function depending only on time. This choice can be fixed by requiring that $\int \Phi_p(t) d\text{vol}_{g_{t,p}} = 0$ for every $(t,p)$. We will calculate the derivative $\partial_t \Delta_{g_{t,p}} \Phi_p(t)$ in two different ways (one using the chain rule and another using the equation for $\Phi_p(t)$). The curvature evolution equation under the normalized Ricci flow is well-known:
    $$\partial_t (k-K_{g_{t,p}}(x))=\Delta_{g_{t,p}}(k-K_{g_{t,p}}(x))+K_{g_{t,p}}(x)(k-K_{g_{t,p}}(x))$$
    and we can re-insert the equation for $\Phi_p(t)$ to get
    $$\Delta_{g_{t,p}}\Delta_{g_{t,p}}\Phi_p(t)-(\Delta_{g_{t,p}}\Phi_p(t))^2+k\Delta_{g_{t,p}}\Phi_p(t)$$
    Now for the chain rule calculation, derivative of the Laplacian is also well-known and gives
    $$\partial_t \Delta_{g_{t,p}} \Phi_p(t)=(K_{g_{t,p}}(x)-k)\Delta_{g_{t,p}}\Phi_p(t)+\Delta_{g_{t,p}}\partial_t\Phi_p(t)$$
    which once again using the equation for $\Phi_p(t)$ can be transformed into
    $$-(\Delta_{g_{t,p}}\Phi_p(t))^2+\Delta_{g_{t,p}}\partial_t\Phi_p(t)$$
    Now comparing the two expressions we deduce that 
    $$\Delta_{g_{t,p}}(\partial_t\Phi_p(t)-\Delta_{g_{t,p}}\Phi_p(t)-k\Phi_p(t))=0.$$
    As harmonic functions on a closed manifold are constant, we have 
    $$\partial_t\Phi_p(t)-\Delta_{g_{t,p}}\Phi_p(t)-k\Phi_p(t)=c(t)$$
    where $c(t)$ is a function depending only on time. We will make use of this calculation in a moment. \par 
    We can use the function $\Phi_p(t)$ to construct another auxiliary function 
    $$H_p(t,x):= K_{g_{t,p}}(x)-k + 2|\nabla^{g_{t,p}}\Phi_p(t,x)|^2$$
    to which we want to apply the maximum principle. We start by calculating its time derivative:
    $$
        \partial_t H_p= \partial_t(K_{g_{t,p}}(x)-k + 2|\nabla^{g_{t,p}}\Phi_p(t,x)|_{g_{t,p}}^2)
    $$
    For the first term, we once again use the curvature evolution equation:
    $$\partial_t (K_{g_{t,p}}(x)-k)=\Delta_{g_{t,p}}(K_{g_{t,p}}(x)-k)+K_{g_{t,p}}(x)(K_{g_{t,p}}(x)-k)$$
    By definition of $\Phi_p$ we can transform the last term into 
    $$K_{g_{t,p}}(x)(K_{g_{t,p}}(x)-k) = (\Delta_{g_{t,p}} \Phi_p(t))^2+k(K_{g_{t,p}}(x)-k)$$
    Now we turn to calculating the time derivative of the $|\nabla^{g_{t,p}}\Phi_p(t,x)|_{g_{t,p}}^2$ term, which we can think of as $g^{-1}_{t,p}(d\Phi_p, d\Phi_p)$:
    \begin{align*}
	    \partial_t |\nabla^{g_{t,p}}\Phi_p(t,x)|_{g_{t,p}}^2&=-(k-K_{g_{t,p}}(x))|\nabla^{g_{t,p}}\Phi_p(t,x)|_{g_{t,p}}^2 \\
								&\quad +2g_{t,p}(\nabla^{g_{t,p}}\Phi_p(t,x), \nabla^{g_{t,p}} \partial_t\Phi_p(t,x))
    \end{align*}
    Using the expression for $\partial_t \Phi_p$ from before, and taking into account that $\nabla^{g_{t,p}} c(t)$ vanishes, we can transform the expression into
    $$(k+K_{g_{t,p}}(x))|\nabla^{g_{t,p}}\Phi_p(t,x)|_{g_{t,p}}^2 +2g_{t,p}(\nabla^{g_{t,p}}\Phi_p(t,x), \nabla^{g_{t,p}} \Delta_{g_{t,p}}\Phi_p(t,x))$$
    Finally, using Bochner's identity to exchange the covariant derivative with the Laplacian
    \begin{align*}
	    &2g_{t,p}(\nabla^{g_{t,p}}\Phi_p(t,x), \nabla^{g_{t,p}} \Delta_{g_{t,p}}\Phi_p(t,x))\\
	    &\qquad =\Delta_{g_{t,p}}|\nabla^{g_{t,p}}\Phi_p(t,x)|_{g_{t,p}}^2-2|\nabla^{2,g_{t,p}}\Phi_p(t,x)|_{g_{t,p}}^2 -K_{g_{t,p}}|\nabla^{g_{t,p}}\Phi_p(t,x)|_{g_{t,p}}^2
    \end{align*}
    we obtain 
    $$k|\nabla^{g_{t,p}}\Phi_p(t,x)|_{g_{t,p}}^2+\Delta_{g_{t,p}}|\nabla^{g_{t,p}}\Phi_p(t,x)|_{g_{t,p}}^2-2|\nabla^{2,g_{t,p}}\Phi_p(t,x)|_{g_{t,p}}^2$$
    Combining both calculations yields 
    \begin{align*}
	    \partial_t H_p &= \Delta_{g_{t,p}}(K_{g_{t,p}}(x)-k) + (\Delta_{g_{t,p}} \Phi_p(t))^2+k(K_{g_{t,p}}(x)-k) \\
			   &\quad +2k|\nabla^{g_{t,p}}\Phi_p(t,x)|_{g_{t,p}}^2+2\Delta_{g_{t,p}}|\nabla^{g_{t,p}}\Phi_p(t,x)|_{g_{t,p}}^2-4|\nabla^{2,g_{t,p}}\Phi_p(t,x)|_{g_{t,p}}^2 \\
			   &= \Delta_{g_{t,p}}H_p+kH_p+(\Delta_{g_{t,p}} \Phi_p(t))^2-4|\nabla^{2,g_{t,p}}\Phi_p(t,x)|_{g_{t,p}}^2
    \end{align*}
    Since $\Delta_{g_{t,p}}$ is the trace of the Hessian $\nabla^{2,g_{t,p}}$, we can estimate the last two terms:
    $$(\Delta_{g_{t,p}} \Phi_p(t))^2-4|\nabla^{2,g_{t,p}}\Phi_p(t,x)|_{g_{t,p}}^2 \leq -2|\nabla^{2,g_{t,p}}\Phi_p(t,x)|_{g_{t,p}}^2 \leq 0$$
    which finally yields the inequality
    $$\partial_t H_p \leq \Delta_{g_{t,p}}H_p+kH_p$$
    This leads us to a direct application of the maximum principle (proposition~\ref{maxprinc}) with $V=0$ and $F(U)=kU$:
    $$K_{g_{t,p}}-k\leq H_p(t) \leq (\underset{x\in F_p}{\max}H_p(0,x)) e^{kt}$$
    Note that we have not claimed to solve the Poisson equation smoothly in the $p$-variable, so we cannot conclude by taking the maximum of $\underset{x\in F_p}{\max}H_p(0,x)$ over $p\in S$ as we do not know that $H_p$ is continuous in $p$. One could probably delve into the theory of 1-parameter Poisson equations on surfaces to prove such result, but we have taken a more straightforward route to find a universal bound. To do so, we proceed by concatenating a series of inequalities:
    \begin{enumerate}
        \item by continuity of $K_{g_{0,p}}-k$, the global maximum among all fibers $F_p$ at $t=0$ exists, we will denote it by $k_0$;
        \item the Schauder estimate for $\Delta^{g_{0,p}}$ implies that there is a constant $A_1$ (uniform because of compactness of $S$ and smoothness of $g_{0,p}$) such that 
        $$\underset{x\in F_p}{\max} |\nabla^{g_{0,p}} \Phi_p(x)|^2 \leq A_1 (\underset{x\in F_p}{\max}|\Phi_p(x)|+||\Delta^{g_{0,p}}\Phi_p||_{\mathcal{C}^{0,\alpha}})^2; $$
        \item the Sobolev embedding $H^2 \to L^\infty$ implies the existence of a constant $A_2$ (uniform because of compactness of $S$ and smoothness of $g_{0,p}$) such that
        $$\underset{x\in F_p}{\max} |\Phi_p(x)| \leq A_2 (||\Phi_p||_2 + ||\Delta^{g_{0,p}}\Phi_p||_2)$$
        recall that $\Delta^{g_{0,p}}\Phi_p = k-K_{g_{0,p}}$ and thus $||\Delta^{g_{0,p}}\Phi_p||_2$ and $||\Delta^{g_{0,p}}\Phi_p||_{\mathcal{C}^{0,\alpha}}$ can be bounded among all fibers $F_p$ using continuity;
        \item using the definition of Laplacian first eigenvalue for mean-value zero functions we can deduce 
        $$||\Phi_p||^2_2 \leq \frac{1}{\lambda_1(g_{0,p})} ||\nabla^{g_{0,p}} \Phi_p||_2^2.$$
        Along with the standard integration by parts trick and  H\"older's inequality
        $$||\nabla^{g_{0,p}} \Phi_p||_2^2 \leq ||\Phi_p||_2 \cdot ||\Delta^{g_{0,p}} \Phi_p||_2,  $$
        we derive the following inequality (after dividing by $|| \Phi_p||_2$): 
        $$||\Phi_p||_2 \leq \frac{1}{\lambda_1(g_{0,p})} ||\Delta^{g_{0,p}} \Phi_p||_2.$$
        where the first eigenvalue of the Laplacian varies continuously with the fibers, and thus has a positive global minimum $\Lambda$.
    \end{enumerate}
    Concatenating these inequalities provides us with the universal bound on $H_p(t=0)$ along all fibers, which we will denote by $H_0^{max}$. This, in turn, gives us the universal bound on the curvature: 
    $$K_{g_{t,p}}-k \leq H^{max}_0 e^{kt}$$
    which combined with the bound from the previous lemma lets us conclude: 
    $$|K_{g_{t,p}}-k| \leq C_0 e^{kt}$$
    for some constant $C_0$ universal for all fibers. 
\end{proof}
Finally, we need to obtain similar estimates for the covariant derivatives of the curvature:
\begin{lem}
    Let $g_{t,p}$ be a solution to the NFRF with surface fibers of genus $g\geq 2$, then for every $i\in \mathbb{N}$ there exists a constant $C_i$ such that for each $(t,p)$ we have:
    $$|\nabla^i (K_{g_{t,p}}-k)|^2 \leq C_i e^{kt/2}$$ 
\end{lem}
\begin{proof}
    The solution essentially consists of repeating the classical argument for the single-surface case (cf., for example, \cite[Section 6, Chapter 5]{ChowKnopf}) and can be split into three steps: the cases $i=1$ and $i=2$, followed by an induction on $i$ using these two base cases. The first two steps are performed by making explicit computations based on the evolution equation of $|\nabla^i K_{g_{t,p}}|^2$:
    $$
        \partial_t|\nabla^{g_{t,p}} K_{g_{t,p}}|^2_{g_{t,p}}=(K_{g_{t,p}}-k)|\nabla^{g_{t,p}} K_{g_{t,p}}|^2_{g_{t,p}}+2g_{t,p}(\nabla^{g_{t,p}} K_{g_{t,p}},\nabla^{g_{t,p}} \partial_t K_{g_{t,p}})
    $$
    We insert the curvature evolution equation in the last term:
    \begin{multline*}
	    \partial_t|\nabla^{g_{t,p}} K_{g_{t,p}}|^2_{g_{t,p}}=(K_{g_{t,p}}-k)|\nabla^{g_{t,p}} K_{g_{t,p}}|^2_{g_{t,p}} \\
								 +2g_{t,p}(\nabla^{g_{t,p}} K_{g_{t,p}},\nabla^{g_{t,p}} (\Delta_{g_{t,p}}(K_{g_{t,p}}-k)+K_{g_{t,p}}(K_{g_{t,p}}-k)))
    \end{multline*}
    We remark that 
    \begin{align*}
        &\nabla^{g_{t,p}} \Delta_{g_{t,p}}(K_{g_{t,p}}-k)=\Delta_{g_{t,p}}\nabla^{g_{t,p}}K_{g_{t,p}}-\frac{1}{2}K_{g_{t,p}}\nabla^{g_{t,p}}K_{g_{t,p}} \\
        &\nabla^{g_{t,p}}(K_{g_{t,p}}(K_{g_{t,p}}-k))=(2K_{g_{t,p}}-k)\nabla^{g_{t,p}}K_{g_{t,p}}
    \end{align*}
    Recall also that 
    $$\Delta_{g_{t,p}} |\nabla^{g_{t,p}} K_{g_{t,p}}|^2_{g_{t,p}}=2g_{t,p}(\Delta_{g_{t,p}}\nabla^{g_{t,p}}K_{g_{t,p}},\nabla^{g_{t,p}}K_{g_{t,p}})+2|\nabla^{2,g_{t,p}} K_{g_{t,p}}|^2_{g_{t,p}}$$
    Regrouping all these calculations gives us the following evolution equation:
    \begin{align*}
	    \partial_t|\nabla^{g_{t,p}} K_{g_{t,p}}|^2_{g_{t,p}} &= (4K_{g_{t,p}}-3k)|\nabla^{g_{t,p}} K_{g_{t,p}}|^2_{g_{t,p}}\\
								 &\quad +\Delta_{g_{t,p}} |\nabla^{g_{t,p}} K_{g_{t,p}}|^2_{g_{t,p}}-2|\nabla^{2,g_{t,p}} K_{g_{t,p}}|^2_{g_{t,p}}
    \end{align*}
    Using the previously established curvature bound, we can bound the first term for $t$ big enough (recall that $k<0$ here):
    \begin{align*}
        (4K_{g_{t,p}}-3k)|\nabla^{g_{t,p}} K_{g_{t,p}}|^2_{g_{t,p}}
        &\leq(k+4H_0^{max}e^{kt})|\nabla^{g_{t,p}} K_{g_{t,p}}|^2_{g_{t,p}}\\
        &\leq \frac{k}{2}|\nabla^{g_{t,p}} K_{g_{t,p}}|^2_{g_{t,p}}
    \end{align*}
    Note that we are able to find $T_0$ independent of the $p$-variable such that this bound is satisfied for $t\geq T_0$ for every fiber. This is true due to the previously established uniformity of the $H_0^{max}$-bound along the fibers. \par 
    Using this, we apply the maximum principle to the inequality
    $$\partial_t|\nabla^{g_{t,p}} K_{g_{t,p}}|^2_{g_{t,p}}\leq \Delta_{g_{t,p}} |\nabla^{g_{t,p}} K_{g_{t,p}}|^2_{g_{t,p}}+\frac{k}{2} |\nabla^{g_{t,p}} K_{g_{t,p}}|^2_{g_{t,p}}$$
    with $V=0$ and $F(U)=kU/2$. This gives us the upper bound on first covariant derivative for $t\geq T_0$:
    $$|\nabla^{g_{t,p}} K_{g_{t,p}}|^2_{g_{t,p}} \leq \underset{p\in S}{\max} \:\underset{F_p}{\max} |\nabla^{g_{T_0,p}} K_{g_{T_0,p}}|^2_{g_{T_0,p}} e^{k(t-T_0)/2}$$
    The bound for $t\in[0;T_0]$ follows from smooth dependence on $p$ and compactness of $S$. \par 
    The calculation for the second covariant derivative is done in essentially the same way: we derive the evolution equation, use the Bochner identities and previously established evolution equations and apply the maximum principle to get an exponential bound for $t\geq T_1$, where $T_1$ is universal for all fibers due to universality of $H_0^{max}$. The details for the single surface case can be found in the Lemma 5.25 in \cite{ChowKnopf}, which are adapted to the fibered case by using the smoothness of $g_{0,p}$ and the universal bound on the curvature. \par 
    The last step consists in using the commutator (we have dropped the coefficients for simplicity)
    $$\nabla^i\Delta K - \Delta\nabla^i K = \sum_{j\leq \lfloor i/2 \rfloor} (\nabla^j K) \underset{g}{\otimes} (\nabla^{i-j} K) $$
    and the time derivative chain rule to derive the corresponding evolution equations, followed by the application of maximum principle for $t$ big enough, getting once again an exponential bound universal among the fibers. The details are rather tedious and technical, and do not add any more clarity to the proof. The single surface case can be found in Proposition 5.27 of \cite{ChowKnopf}, and by using the universal character of $H_0^{max}$ and compactness of $S$ it is easily adapted to the fibered case.
\end{proof}
Finally, using these estimates and the compactness of $S$, we can conclude:
\begin{prop}
    Let $g_{t,p}$ be a solution to the NFRF with surface fibers of genus $g\geq 2$, then the limit $g_{\infty, p}$ is a smooth family of constant curvature metrics.
\end{prop}
\begin{proof}
    The proof consists of repeating the single-surface case word-by-word. Using the smoothness of the initial data $g_{0,p}$ in the $p$-variable and the universality of curvature estimates found in the previous lemmas, we can establish the that this family is continuous. Smoothness would follow from similar estimates for $p$-derivatives $\partial^j_p g_{t,p}$ uniform in $t$, which could be derived in a similar manner. Alternatively, one can use the uniqueness of metric with given constant curvature in a conformal class together with result in \cite{EE}.
\end{proof} 
This proposition essentially concludes the proof of theorem~\ref{NFRF}. We are now ready to combine Lemma~\ref{slicing} and Theorem~\ref{NFRF} to establish the existence of classifying maps:
\begin{thm}\label{thm:main1}
    Let $f:X\to S$ be an achiral Lefschetz fibration of fiber genus $g\geq 4$. Denote $p_1,\dots,p_n$ its critical values and $\tau_1,\dots,\tau_n$ corresponding monodromies. Then there exists a classifying map $\phi_f: S\to \dm_g$, which is unique up to a choice of slicing metrics for $\tau_1,\dots,\tau_n$ and homotopy of $\phi_f|_{\sreg}: \sreg\to \m_g$. \par 
    Moreover, if $f$ does not admit $(1,g-1)$-type separating singularities, we can arrange the image of $\phi_f$ to be contained in the manifold $\dm_g^o$, the automorphism-free locus of $\dm_g$, and $\phi_f^*(\scr_g) \simeq X$.
\end{thm} 
\begin{proof}
    We start by choosing slicing metrics $g_{\tau_i}$ corresponding to families of automorphism-free hyperbolic metrics on the disk $D_i$ for each singular value $p_i$, this is always possible by Proposition~\ref{slicingprime}, choosing $[C_0]$ automorphism-free and perturbing the disk if necessary (using $g\geq 4$), except when the vanishing cycle is separating of type $(1,g-1)$ (since $\Delta^o_{1,g-1}=\emptyset$). \par 
    The lemma~\ref{slicing} ensures that any such choice can be extended to a metric $g_0$ on $X\setminus\text{Crit}$, which is unique up to homotopy. Moreover, as the real codimension of metrics which are conformally equivalent to a hyperbolic structure with non-trivial automorphisms is $2(g-2)$, for $g\geq 4$ a generic extension $g_0$ is conformally equivalent to an automorphism-free hyperbolic metric on each fiber, by the parametric transversality theorem and since the flow preserves the conformal classes. \par 
    Now consider the NFRF with initial condition given by $g_{0,p}:= g_0|_{F_p}$ for $p\in \sreg$ and normalization $k=-2$, after rescaling each $g_{0,p}$ to area $4\pi(g-1)$. The compactness of $\sreg$ and fiber genus hypothesis imply the existence of the solution $g_{t,p}$ to NFRF and its convergence to a constant curvature family of metrics $g_{\infty, p}$ which is smooth in the $p$-variable. Moreover, as the NFRF with $k=-2$ preserves metrics of constant Gauss curvature $-1$, it will not modify the slicing metrics on the overlap $\ssing\cap \sreg$, and thus the family defined by
    $$g_{p}=\begin{cases}
        g_{\infty,p} &\text{for }p\in \sreg \\
        g_{\tau_i} &\text{for }p\in D_i\subseteq \ssing
    \end{cases}$$
    is a smooth family of constant negative curvature Riemannian metrics on the fibers $F_p$ for $p\notin \{p_1,\dots, p_n\}$. This is sufficient to define a smooth classifying map $\phi_f : S\to \dm_g$. If moreover all fibers are automorphism-free, the
identification of each fiber $F_p$ with the corresponding hyperbolic surface is unique. The identifications of $\xreg$ and of the $f^{-1}(D_i)$ (Lemma~\ref{slicing}) with the pullback of $\scr_g$ therefore agree on the overlaps and glue to $\phi_f^*(\scr_g)\simeq X$. Moreover, we can always arrange this map to be locally injective by perturbing the choice of initial metrics $g_0$ if needed.
\end{proof}
\begin{rem}
    This theorem can be generalized to genera 2 and 3, however, a more careful approach is required, either by rigidifying the moduli space (by means of level structures, see e.g. \cite{BP} and \cite{Looijenga}), or by passing to finite covers (which corresponds to introduction of additional structure on the fibers).
\end{rem}
\section{Signature of achiral Lefschetz fibrations}\label{section:signature}
One of the main goals of the classifying map construction is to extract invariants of $X$. It turns out that the signature of $X$ can be easily calculated using the cohomology of the moduli space $\dm_g$:
\begin{thm}\label{thm:main2}
    Let $f:X\to S$ be an achiral Lefschetz fibration of fiber genus $g\geq 4$ with a classifying map $\phi_f : S\to \dm_g$. Then we have:
    $$\sigma(X) = \langle \hat\sigma, [\phi_f]\rangle$$
    where $\hat{\sigma} = 4c_1(\lambda)-PD(\sdiv_g)\in H^2(\dm_g,\mathbb{Q})$. 
\end{thm}
\begin{rem}
\begin{enumerate}
    \item We use rational coefficients to be able to use the Poincaré duality. 
    \item There is a subtle point in distinguishing homotopy of $\dm_g$ seen as an orbifold and seen as a topological space. We always consider the orbifold homotopy type of this space, which is justified since we can construct our classifying map as a map of orbifolds.
\end{enumerate}
\end{rem}
This formula is a generalization of the following result by Smith, see~\cite{Smith1}:
\begin{thm}[Smith]
    Let $f:X\to S$ be a Lefschetz fibration giving rise to a classifying map $\phi_f: S\to \dm_g$. Then we have 
    $$\sigma(X) = \langle4c_1(\lambda), [\phi_f]\rangle - \delta$$
    where $\delta$ is the number of vanishing cycles of $f$.
\end{thm}
The original proof can be adapted without any major changes. Smith uses Novikov additivity to cut the base into regular and singular parts. The only difference between the chiral and achiral cases appears when evaluating signature of the singular part, which gives exactly the signed intersection count part. The signature of the regular part is then calculated using the Atiyah--Singer generalization of the Hirzebruch signature theorem. \par 
We present here another simple proof which has the advantage of not using the index theorem directly, and thus can potentially be more accessible. This proof resembles the approach of Endo in \cite{Endo} and the initial considerations of Meyer in \cite{Meyer}. We start by recalling the definition of signature cocycle $\tau_g$:
\begin{defin}
    Let $S$ be a three-holed sphere (or a pair of pants), and $X$ be an oriented $\Sigma_g$-bundle over $S$ with boundary monodromies given by $\alpha, \beta, \beta^{-1}\alpha^{-1} \in \mo_g$. The function 
    $$\tau_g : \mo_g \times \mo_g \to \mathbb{Z}$$
    defined as $\tau_g(\alpha, \beta):= \sigma(X, \partial X)$ is called the Meyer signature cocycle.
\end{defin}
\begin{thm} [\cite{Meyer}]
    The signature cocycle $\tau_g$ is well defined. Moreover, its action extends to 2-chains representable by compact surfaces in the following sense: for $c\in C_2(B\mo_g, \mathbb{Z})$ a singular 2-chain, choose a surface $S$ (possibly with boundary) representing it. Then if $S = \cup S_i$ is a decomposition of $S$ into pairs of pants, we have $\tau_g(S) = \sum \tau_g(S_i)$. This formula is independent of the choice of a pants decomposition as 
    $$\tau_g(S) = \sum \tau_g(S_i) = \sum \sigma(X_i, \partial X_i) = \sigma(X, \partial X)$$
    For surfaces not admitting a pair of pants decomposition ($S^2, \Sigma_0^1,\Sigma_0^2, T$) we extend by $0$.
\end{thm}
\begin{rem}
    The function $\tau_g$ is called a cocycle because it is a group cocycle; its class lies in the classifying space cohomology $H^2(B\mo_g, \mathbb{Z})$. \par 
    Remark also that this class can be thought of as initially constructed on $BSp(2g,\mathbb{Z})$ and then pulled back via the morphism $\mo_g \to Sp(2g,\mathbb{Z})$.
\end{rem}
Now that we know how to calculate the signature of the regular part $\xreg$ using the Meyer cocycle, we want to tie this calculation with the Chern class of the Hodge bundle:
\begin{lem}
    Let $\tau_g \in C^2(B\mo_g, \mathbb{Q})$ be the image of the integral Meyer cocycle $\tau_g$. Then there exist a 2-cocycle $\alpha \in C^2(\dm_g, \mathbb{Q})$ and a 1-cochain $\beta \in C^1(\m_g, \mathbb{Q})$ such that:
    \begin{align*}
        &[\alpha] = 4c_1(\lambda) \in H^2(\dm_g, \mathbb{Q}) \\
        &\langle c, \iota^*\alpha + d\beta\rangle = \sigma(M_c,\partial M_c),  \forall c \in C_2(\m_g,\mathbb{Q}) \text{ representable by a compact surface,}
    \end{align*}
    where $\iota^*$ is the restriction morphism induced by the inclusion $\iota : \m_g \to \dm_g$, and $M_c$ is a fiber bundle above a surface $S_c$ representing the 2-chain $c$.
\end{lem}
\begin{proof}
    We start by recalling that due to the results of Morita \cite{charclasses} and Meyer \cite{Meyer} we know that $\psi(\tau_g) = 4c_1(\lambda)$ in $H^2(\m_g, \mathbb{Q})$, where $\psi$ is the natural isomorphism in cohomology $\psi : H^*(B\mo_g, \mathbb{Q}) \overset{\simeq}{\to} H^*(\m_g, \mathbb{Q})$. \par 
    Note now that via its action on 2-chains in $C_2(B\mo_g, \mathbb{Q})$, we can also transport the Meyer signature cocycle $\tilde{\psi}(\tau_g) \in C^2(\m_g, \mathbb{Q})$ to the cochains of the moduli space using the map $\tilde{\psi} : C^2(B\mo_g, \mathbb{Q}) \to C^2(\m_g, \mathbb{Q})$. \par 
    Finally, let $\alpha \in C^2(\dm_g, \mathbb{Q})$ be a cochain representing the class $4c_1(\lambda) \in H^2(\dm_g, \mathbb{Q})$ in cohomology (note that here we talk about the extension of $c_1(\lambda)$ to the Deligne--Mumford compactification). This implies that 
    $$[\iota^*\alpha] = 4\iota^*c_1(\lambda) =[\Tilde{\psi}(\tau_g)] \in H^2(\m_g, \mathbb{Q}), $$
    and thus there exists a 1-cochain $\beta \in C^1(\m_g, \mathbb{Q})$ such that $\Tilde{\psi}(\tau_g) = \iota^*\alpha + d\beta$.
\end{proof}
    As a crucial component, we need the following special case of Smith's theorem. We use his intermediate result on projective fibrations, which boils down to explicit calculations in the framework of complex algebraic geometry:
\begin{lem}[Corollary 4.6 in \cite{Smith1}] \label{SmithsProjFibr}
    Let $f:X \to B$ be a holomorphic Lefschetz fibration on a smooth projective surface, then $$\sigma(X) = \langle 4c_1(\lambda), [\phi_f] \rangle - \delta, $$
    where $\delta$ is the number of singular fibers.
\end{lem}
    With this result in hand, we are almost ready to assemble our proof of Theorem~\ref{thm:main2}. One key difference with the original proof of Smith is that we need this formula to be true not just once per genus, but also once per genus and per each singularity type. We prove the existence of such Lefschetz fibrations with a direct construction:
\begin{lem}\label{lefex}
	Let $g\geq 2$. 
	\begin{enumerate}
		\item There is a holomorphic Lefschetz fibration $f_0:X_0\to \mathbb{C}P^1$ of genus $g$ with at least one singularity, such that all of its vanishing cycles are non-separating.
		\item For every $1 \leq h \leq \lceil g/2 \rceil$ there is a holomorphic Lefschetz fibration $f_h:X_h\to \mathbb{C}P^1$ of genus $g$ with at least one singularity, such that exactly one vanishing cycle is separating of type $(h,g-h)$ and the rest are non-separating.
	\end{enumerate}
\end{lem}
\begin{proof}
	We will construct our examples as branched covers of the family $\text{pr}: \mathbb{C}P^1\times \mathbb{C}P^1 \to \mathbb{C}P^1$. Denote its total space by $Q$, the class of a fiber by $F$ and the class of a constant section by $s$. 
	\begin{enumerate}
		\item Consider an effective divisor $B$, which is a general member of a linear system $|(2g+2)F+2s|$. According to Bertini's theorem, it is represented by a smooth complex curve in $Q$ and $\text{pr}|_B: B\to \mathbb{C}P^1$ is of degree $B\cdot F = 2g+2$ and only has simple ramified points. Let $\pi: X_0 \to Q$ be the double branched cover along $B$ and $\text{pr}':X_0 \to \mathbb{C}P^1$ be the composition $\text{pr}\circ \pi$. This is a Lefschetz fibration (with Lefschetz critical points corresponding to the simple ramification points), singular fibers are irreducible curves and genus of a regular fiber is precisely $g$. Moreover, since $g\geq 2$, we should have at least one ramification point and thus at least one Lefschetz singularity.
		\item Construction for a reducible fiber is more involved, so we will describe the outline while dropping some technical details. Let $Q_1$ be the blowup of $Q$ at $(0,0)$ and $\Tilde{Q}$ be a blowup of $Q_1$ at the horizontal direction at $(0,0)$. Denote the corresponding exceptional spheres by $E_1$ and $E_2$ respectively. \par 
		Consider a linear system $|(2g+2)f+2s-(2h+1)E_1 - (2h+1)E_2|$. One can show that it is basepoint-free, and thus by Bertini's theorem we can choose a smooth representative $B_h$. Consider also a divisor $E'=E_1-E_2$, which can be represented by a smooth curve. Simple calculation shows that $B$ and $E_1'$ are disjoint, and the divisor $R=B+E'$ is even, so we can take a double branched cover $\pi:\Tilde{X}_h\to \tilde{Q}$ along $R$. \par 
		It can be shown that outside of $0$ the smooth fibers are curves of genus $g$ and singular fibers are irreducible with a single node. The fiber at $0$ is, however, a disjoint union of a genus $h$ curve $A_1'$ and a genus $g-h$ curve $A_2'$. However, each of these curves intersect $\Tilde{E}:= \frac{1}{2} \pi^*(E')$ at a single point. Moreover, a straightforward calculation shows that $\Tilde{E}^2 = -1$, so we can contract this sphere in $\Tilde{X}_h$ to obtain a holomorphic map $p:X_h\to \mathbb{C}P^1$. After contracting the sphere, the fiber at $0$ becomes a reducible curve with components of genera $h$ and $g-h$ respectively, whereas all other singular fibers are irreducible curves. Careful calculation shows that this is indeed a Lefschetz fibration.
	\end{enumerate}

\end{proof}
\begin{proof}[Proof of Theorem~\ref{thm:main2}]
    We start by calculating the signature of the regular part $\xreg$. From Meyer's theorem and its adaptation to the moduli space chains, we know that 
    $$\sigma(\xreg, \partial \xreg) = \langle \phi_f(\sreg), \Tilde{\psi}(\tau_g) \rangle.$$
    Coupled with the decomposition from the previous lemma, we can write it as 
    $$\sigma(\xreg, \partial\xreg) = \langle \phi_f(\sreg), \iota^*\alpha\rangle + \langle \phi_f(\sreg), d\beta\rangle.$$
     Since $\partial\sreg=-\sqcup\partial D'_i$, Novikov additivity applied to $X= (\xreg, \partial \xreg) \cup \left(\sqcup (D'_i, \partial D'_i) \right)$ together with $\langle \phi_f(\sreg), \iota^*\alpha\rangle = \langle 4c_1(\lambda), [\phi_f]\rangle - \sum \langle \phi_f(D'_i), \alpha\rangle $, gives
    $$\sigma(X)=\langle 4c_1(\lambda),[\phi_f]\rangle+\sum_i c(D'_i),\; c(D):=\sigma(f^{-1}(D))-\langle\phi_f(D),\alpha\rangle-\langle\phi_f(\partial D),\beta\rangle .$$
    Let $E,E'$ be slicing disks meeting the one-node locus $\Delta^1$ of the same irreducible component of $\sdiv_g$ (this one-node locus is connected)  with the same sign. Join $\partial E$ to $\partial E'$ by an annulus $A$ in a tubular neighbourhood $N$ of $\Delta^1$, avoiding $\Delta$: transport the meridian along a path in $\Delta^1$. Then $E-E'-A$ is a sphere that is null-homologous in $N\simeq\Delta^1$, so $\langle E,\alpha\rangle-\langle E',\alpha\rangle=\langle A,\alpha\rangle$. Bundles over an annulus have signature $0$, so $\langle\partial E-\partial E',\beta\rangle=\langle A,d\beta\rangle=-\langle A,\alpha\rangle$. Hence $c(E)=c(E')$. Similarly, a conjugate (achiral) disk satisfies $c(\bar{E})=-c(E)$: indeed $\partial\bar{E}=-\partial E$, the sphere $E \cup \bar{E}$ is null-homologous in $N$ and $f^{-1}(\bar{E})$ is $f^{-1}(E)$ with opposite orientation. Thus $c$ only depends on the component of $\sdiv_g$ and on the sign. \par 
    The result for the component $\Delta_{irr}$ comes from the part 1 of lemma \ref{lefex} together with lemma \ref{SmithsProjFibr}: we get $c(E)=-1$ for chiral and $c(E)=+1$ for achiral slicing disks. The result for components $\Delta_{h,g-h}$ for $1\leq h \leq \lceil g/2 \rceil$ comes from the part 2 of lemma \ref{lefex}, lemma \ref{SmithsProjFibr} and the already established result for the component $\Delta_{irr}$: we get the same answer, $\mp 1$ depending on chirality. Combining these results we deduce $\sum c(D'_i) = -\langle PD(\sdiv_g), [\phi_f]\rangle$, which is the claimed formula.
\end{proof}
\bibliographystyle{plain}
\bibliography{7-Yakupov}
\end{document}